\documentclass[11pt,reqno]{amsart}
\usepackage{mathrsfs}
\usepackage[breaklinks]{hyperref}
\usepackage[headheight=110pt,top=1.0in, bottom=1.0in, left=0.8in, right=0.8in]{geometry}

\usepackage{url}
\usepackage{amssymb}
\usepackage{amsmath}
\allowdisplaybreaks
\newcommand{\pFq}[5]{\ensuremath{{}_{#1}F_{#2} \left( \genfrac{}{}{0pt}{}{#3}
{#4} \bigg| {#5} \right)}}

\renewcommand{\Re}{\operatorname{Re}}

\newcommand{\G}{\Gamma}

\renewcommand{\(}{\left\(}
\renewcommand{\)}{\right\)}
\renewcommand{\[}{\left\[}
\renewcommand{\]}{\right\]}
\newcommand{\Z}{\mathbb{Z}}
\newcommand{\K}{\mathbb{K}}
\newcommand{\Q}{\mathbb{Q}}
\newcommand{\C}{\mathbb{C}}
\newcommand{\R}{\mathbb{R}}

\let\dotlessi=\i
\renewcommand{\i}{\infty}
\numberwithin{equation}{section}
 \theoremstyle{plain}
\newtheorem{theorem}{Theorem}[section]
\newtheorem{lemma}[theorem]{Lemma}
\newtheorem{corollary}[theorem]{Corollary}
\newtheorem{proposition}[theorem]{Proposition}

   \makeatletter
\def\proof{\@ifnextchar[{\@oproof}{\@nproof}}
\def\@oproof[#1][#2]{\trivlist\item[\hskip\labelsep\textit{#2 Proof of\
#1.}~]\ignorespaces}
\def\@nproof{\trivlist\item[\hskip\labelsep\textit{Proof.}~]\ignorespaces}

\makeatother

\makeatletter
\def\@tocline#1#2#3#4#5#6#7{\relax
  \ifnum #1>\c@tocdepth 
  \else
    \par \addpenalty\@secpenalty\addvspace{#2}%
    \begingroup \hyphenpenalty\@M
    \@ifempty{#4}{%
      \@tempdima\csname r@tocindent\number#1\endcsname\relax
    }{%
      \@tempdima#4\relax
    }%
    \parindent\z@ \leftskip#3\relax \advance\leftskip\@tempdima\relax
    \rightskip\@pnumwidth plus4em \parfillskip-\@pnumwidth
    #5\leavevmode\hskip-\@tempdima
      \ifcase #1
       \or\or \hskip 1em \or \hskip 2em \else \hskip 3em \fi%
      #6\nobreak\relax
    \dotfill\hbox to\@pnumwidth{\@tocpagenum{#7}}\par
    \nobreak
    \endgroup
  \fi}
\makeatother

\usepackage{bigints}
\usepackage{suffix}
\usepackage{mathtools}
\DeclarePairedDelimiterX\MeijerM[3]{\lparen}{\rparen}%
{\begin{smallmatrix}#1 \\ #2\end{smallmatrix}\delimsize\vert\,#3}

\newcommand\MeijerG[8][]{%
  G^{\,#2,#3}_{#4,#5}\MeijerM[#1]{#6}{#7}{#8}}

\WithSuffix\newcommand\MeijerG*[7]{%
  G^{\,#1,#2}_{#3,#4}\MeijerM*{#5}{#6}{#7}}

\numberwithin{theorem}{section}
\numberwithin{equation}{section}

\begin{document}
\title[Number field analogue of divisor function \`a la Koshliakov]{Number field analogue of divisor function \`a la Koshliakov}
\author{Soumyarup Banerjee and Rahul Kumar}\thanks{2010 \textit{Mathematics Subject Classification.} Primary 11M06, 33E20; Secondary 33C10.\\
\textit{Keywords and phrases.} Number field, divisor function, Dedekind zeta function, Koshliakov kernel}
\address{ Discipline of Mathematics, Indian Institute of Technology Gandhinagar, Palaj, Gandhinagar 382355, Gujarat, India}\email{soumyarup.b@iitgn.ac.in;  rahul.kumr@iitgn.ac.in}

\begin{abstract}
In this article, we study a divisor function in an arbitrary number field akin to Koshliakov's work on Vorono\"{\dotlessi} summation formula. More precisely, we generalize Koshliakov's kernel and Koshliakov's transform over any number field to obtain identities for the Lambert series associated to the divisor function in an arbitrary number field.
\end{abstract}
\maketitle
\vspace{-0.8cm}

\section{Introduction}\label{intro}
The asymptotic behaviour of an arithmetic function has long been a fascinating subject in analytic number theory. In particular, one often investigates the behaviour of $a(n)$ as $n$ increases. A common technique to understand arithmetic functions involves studying the partial sums $\sum_{n\leq x}a(n)$. For example, Dirichlet famously estimated the asymptotic behaviour of the partial sums $\sum_{n\leq x} d(n)$ by relating it to the problem of counting the number of lattice points lying inside or on the hyperbola, where $d(n)$ denotes the divisor function i.e, $d(n) = \sum_{d\mid n} 1$. He obtained an asymptotic formula with the main term $x\log x + (2\gamma - 1)x$, where $\gamma$ is the Euler's constant and an error term of order $\sqrt{x}$. The problem of estimating the error term is known as the Dirichlet hyperbola problem or the Dirichlet divisor problem. The bound on the error term has been further improved by many mathematicians. At this writing, the best estimate $O(x^{131/416+\epsilon})$, for each $\epsilon > 0$, as $x \to \infty$, is due to M.N. Huxley \cite{Huxley}. On the other hand, G.H. Hardy \cite{Hardy} has shown that the error term can not be equal to $O(x^{1/4})$, with the best result in this direction currently due to K. Soundararajan \cite{Sound}. It is conjectured that the error term may be written as $O(x^{1/4+\epsilon})$ for each $\epsilon > 0$, as $x \to \infty$.

The Dirichlet divisor problem can be naturally generalized by considering the function $\sigma_a(n) : = \sum_{d\mid n} d^a$, where $a$ is any complex number. The problem of determining the correct order of magnitude of the error term as $x \to \infty$ for the partial sum $\sum_{n\leq x} \sigma_a(n)$ is known as the extended divisor problem \cite{Lau}. In this case, the main term can be evaluated as $\zeta(1-a) x+ \frac{\zeta(1+a)}{1+a}x^{1+a}$, where $\zeta(s)$ denotes the Riemann zeta function. It is conjectured that for each $\epsilon >0$, the error term may be written as $O(x^\epsilon)$ for
$-1< a \leq-1/2$ and  $O(x^{1/4+a/2+\epsilon})$ for $-1/2\leq a < 0$ as $x \to \infty$.

Vorono\"{\dotlessi} \cite{Voronoi} introduced a new phase of study into the Dirichlet divisor problem. He was able to express the error term as an infinite series involving the Bessel functions which are defined in \S \ref{sec:specialfunctions}. More precisely, letting $Y_\nu$ (resp. $K_\nu$) denote the Bessel function of the second kind (resp. modified Bessel function of second kind) of order $\nu$ and $\gamma$ denote the Euler constant, a celebrated identity of Vorono\"{\dotlessi} is given by
\begin{equation}\label{Voronoi 1-identity}
\sideset{}{'}\sum_{n\le x}\!\!d(n) =x\log x+(2\gamma-1)x+\frac{1}{4} - \sqrt{x}\sum_{k=1}^\infty\frac{d(k)}{\sqrt{k}}
\left(Y_1\left(4\pi\,\sqrt{xk}\,\right)+\frac{2}{\pi}K_1\left(4\pi\,\sqrt{xk}\,\right)\right),
\end{equation}
where $\sum'$ indicates that only $\frac{1}{2}d(n)$ is counted when $x$ is an integer. In the same article \cite{Voronoi}, Vorono\"{\dotlessi} also obtained a more general form of \eqref{Voronoi 1-identity}, namely
\begin{equation}\label{Voronoi 2-identity}
\sum_{\alpha<n<\beta} d(n) f(n) = \int_\alpha^\beta (2\gamma + \log t) f(t) {\rm d}t + 2\pi \sum_{n=1}^\infty d(n) \int_\alpha^\beta f(t) \left(\frac{2}{\pi} K_0(4\pi \sqrt{nt}) - Y_0(4\pi \sqrt{nt}) \right) {\rm d}t,
\end{equation}
where $f(t)$ is a function of bounded variation in $(\alpha, \beta)$ and $0<\alpha<\beta$. A shorter proof of the above identity for $0<\alpha<\beta$ with $\alpha, \beta \not\in \Z$ was offered by Koshliakov in \cite{Koshliakov}, where he assumed $f$ any analytic function lying inside a closed contour strictly containing the interval $[\alpha, \beta]$. The identity \eqref{Voronoi 2-identity} was further generalized for the general divisor function in \cite[Section 6, 7]{bdrz}, which precisely states that for $0<\alpha<\beta$ with $\alpha, \beta \not\in \Z$ and $-\frac{1}{2} < \Re (s) < \frac{1}{2}$,
\begin{align}\label{Voronoi identity for sigma_a}
\sum_{\alpha<n<\beta} &\sigma_{-s}(n)f(n) = \int_{\alpha}^{\beta} \left(\zeta(1+s)+t^{-s}\zeta(1-s)\right) f(t) \, {\rm d}t +2\pi \sum_{n=1}^{\infty} \sigma_{-s}(n)n^{s/2}  \nonumber \\
&\times \int_{\alpha}^{\beta} t^{-\frac{s}{2}} \left \lbrace \left( \frac{2}{\pi} K_s(4\pi \sqrt{nt}) - Y_s(4\pi \sqrt{nt}) \right) \cos\left( \frac{\pi s}{2}\right) - J_s(4\pi \sqrt{nt}) \sin\left( \frac{\pi s}{2}\right)  \right \rbrace f(t) \, {\rm d} t 
\end{align}
where $f$ denotes a function which is analytic inside a closed contour strictly containing $[\alpha, \beta]$. 

Throughout the paper, we let our number field be $\K$ with extension degree $[\K : \Q] = d$ and signature $(r_1, r_2)$ (i.e., $d = r_1+2r_2$). Let $d_\K$ denotes the discriminant of $\K$ with absolute value $D_\K$. Let $\mathcal{O}_\K$ be its ring of integers and $v_{\K}(m)$ denotes the number of non-zero integral ideals in $\mathcal{O}_K$ with norm $m$. Let $\mathfrak{N}$ be the norm map of $\K$ over $\Q$ and $\mathfrak{N}_{\K/\Q}(I)$ denotes the absolute norm of an non-zero integral ideal $I \subseteq \mathcal{O}_\K$.  We generalize the divisor function for an arbitrary number field $\K$ as
$$\sigma_{\K, a}(n) = \sum_{\substack{I \subseteq \mathcal{O}_\K \\  \mathfrak{N}_{\K/\Q}(I) \mid n}} \left( \mathfrak{N}_{\K/\Q}(I)\right)^a$$
where $a$ is any complex number. Note that for $\K = \Q$, we have $\sigma_{\K, a}(n) = \sigma_a(n)$. We denote the Dedekind zeta function over $\K$ as $\zeta_\K(s)$ which is defined later in \S \ref{Dedekind}.

Our first goal here is to find a Vorono\"{\dotlessi}-type identity for the divisor function $\sigma_{\K, a}(n)$ over any number field $\K$. Let $(v)_\ell$ denotes an $\ell$-tuple with all of its entries equal to $v$. For any $\nu \in \C$ and $x \in \C\setminus \lbrace 0 \rbrace$, we define a function over any number field $\K$ of degree $d$ with signature $(r_1, r_2)$ as
\begin{equation}\label{Our Kernel}
G_{\K, \nu}(x) := \MeijerG*{d+1}{0}{0}{2(d+1)}{-}{-\frac{\nu}{2}, (\frac{\nu}{2})_{r_1+r_2}, \left(\frac{1+\nu}{2}\right)_{r_2}; \frac{1-\nu}{2}, (\frac{\nu}{2})_{r_2}, \left(\frac{1+\nu}{2}\right)_{r_1+r_2}}{\frac{x^2}{16}}
\end{equation}
 where the function on the right hand side denotes the ``Meijer G-function"  which is defined later in \S \ref{sec:specialfunctions}. The following theorem provides the Vorono\"{\dotlessi}-type identity for $\sigma_{\K, a}(n)$ over any number field $\K$. 
\begin{theorem}\label{zetazetathm}
Let $H$ be the residue of $\zeta_\K(s)$ at $s=1$. For any complex number $a$ with $-\frac{1}{2}<\mathrm{Re}(a)<\frac{1}{2}$ and any Schwarz function $f$, we have
\begin{align}\label{zetazetathmeqn}
\sum_{n=1}^\infty\sigma_{\K,a}(n)f(n)=\int_0^\infty \left(\zeta_\K(1-a)+ t^a \zeta(1+a) H \right)  f(t) \, {\rm d}t -\frac{1}{2} \zeta_\K(-a)f(0^+) \nonumber\\
+2 \pi^{ \frac{1+a +(1-a)d}{2}}D_{\K}^{\frac{a-1}{2}}\sum_{n=1}^\infty \sigma_{\K,-a}(n)n^{a/2} \int_0^\infty t^{a/2}G_{\K, a/2}\left(\frac{4\pi^{(d+1)}nt}{D_\K}\right)f(t)\ {\rm d}t,
\end{align}
provided the Mellin transform of $f$ decays faster than any polynomial in any bounded vertical strip.
\end{theorem}
\noindent
Koshliakov in \cite{Koshliakov2} investigated integrals involving the following kernels
\begin{equation}\label{First Koshliakov kernel}
\cos(\pi \nu)\left\lbrace \frac{2}{\pi}K_{2\nu}(2\sqrt{xt})-Y_{2\nu}(2\sqrt{xt})  \right\rbrace - \sin(\pi \nu)J_{2\nu}(2\sqrt{xt})
\end{equation} 
and
\begin{equation*}
 \cos(\pi \nu)\left\lbrace \frac{2}{\pi}K_{2\nu}(2\sqrt{xt})+Y_{2\nu}(2\sqrt{xt})  \right\rbrace + \sin(\pi \nu) J_{2\nu}(2\sqrt{xt})
\end{equation*}
where the first one is called the first Koshliakov kernel in \cite[p. 897]{bdrz}. It has significant importance in the Vorono\"{\dotlessi} summation formula since it occurs naturally in the last term of \eqref{Voronoi identity for sigma_a}.
 
In this article, we generalize first Koshliakov kernel over any number field $\K$ of degree $d$ with signature $(r_1, r_2)$ by considering the kernel $G_{\K, \nu}(xt)$. The following theorem provides the special cases of the generalized kernel.
\begin{theorem}\label{Special case of new kernel}
Let $x$ and $t$ be any non-zero complex numbers.
\begin{itemize}
\item[(1)]
For $\K = \Q$, we have
$$G_{\K, \nu}(xt) = \cos(\pi \nu)\left\lbrace \frac{2}{\pi}K_{2\nu}(2\sqrt{xt})-Y_{2\nu}(2\sqrt{xt})  \right\rbrace - \sin(\pi \nu)J_{2\nu}(2\sqrt{xt}).$$
\item[(2)]
For $\K$ any imaginary quadratic number field, we have
\begin{align*}
G_{\K, \nu}(xt) = \frac{\sqrt{\pi}}{\sin(2\pi\nu)}&\Bigg\{\frac{2^{1-4\nu}}{\Gamma^2(1-2\nu)}\left(\frac{xt}{4}\right)^{-\nu}\pFq05{-}{1-\nu,1-\nu,\frac{1}{2}-\nu,\frac{1}{2}-\nu,\frac{1}{2}}{-\frac{x^2 t^2}{16}}\nonumber\\
&-\frac{2^{1+2\nu}\cos(\pi\nu)}{\Gamma(1+2\nu)}\left(\frac{xt}{4}\right)^{\nu}\pFq05{-}{1+\nu,\frac{1}{2}+\nu,\frac{1}{2},\frac{1}{2},1}{-\frac{x^2 t^2}{16}}\nonumber\\
&-\frac{2^{4+2\nu}\sin(\pi\nu)}{\Gamma(2+2\nu)}\left(\frac{xt}{4}\right)^{1+\nu}\pFq05{-}{\frac{3}{2}+\nu,1+\nu,\frac{3}{2},\frac{3}{2},1}{-\frac{x^2 t^2}{16}}\Bigg\},
\end{align*}
\end{itemize}
where ${}_0F_5$ denotes the generalized hypergeometric function which is defined in \S \ref{sec:specialfunctions}.
\end{theorem}
The following corollary follows immediately from Theorem \ref{zetazetathm} by substituting the value of the kernel from Theorem \ref{Special case of new kernel} for $\K= \Q$.

\begin{corollary}\label{voronoiforsigma}
Let $-\frac{1}{2}<\mathrm{Re}(a)<\frac{1}{2}$ and $f$ be any Schwarz function. Then we have
\begin{align*}
&\sum_{n=1}^\infty \sigma_{a}(n)f(n)=\int_0^\infty \left(\zeta(1-a)+ t^a \zeta(1+a) \right)  f(t) \, {\rm d}t -\frac{1}{2} \zeta(-a)f(0^+)\nonumber\\
&+2\pi\sum_{n=1}^\infty \sigma_{a}(n)n^{-\frac{a}{2}}\int_0^\infty y^{\frac{a}{2}}f(y)\Bigg\{\left(\frac{2}{\pi}K_a(4\pi\sqrt{ny}-Y_a(4\pi\sqrt{ny}))\right) \cos\left(\frac{\pi a}{2}\right) -J_a(4\pi\sqrt{ny})\sin\left(\frac{\pi a}{2}\right)\Bigg\}\ dy.
\end{align*}
provided the Mellin transform of $f$ decays faster than any polynomial in any bounded vertical strip.
\end{corollary}

\noindent
The integral Involving the first Koshliakov kernel of the form
$$\int_0^\infty f(t, \nu) \left \lbrace \cos(\pi \nu)\left( \frac{2}{\pi}K_{2\nu}(2\sqrt{xt})-Y_{2\nu}(2\sqrt{xt})  \right) - \sin(\pi \nu)J_{2\nu}(2\sqrt{xt}) \right\rbrace {\rm d}t$$
is sometimes known as the first Koshliakov transform of $f(t, \nu)$ (cf. \cite{bdrz}). Some of the integrals of the above form for $\nu = 0$ were previously investigated by Dixon and Ferrar \cite{DF}. In the same article \cite[Equation (8)]{Koshliakov2}, Koshliakov proved a remarkable result that $K_\nu(t)$ is self-reciprocal in the first Koshliakov kernel. In other words, for $-\frac{1}{2} < \Re (\nu) < \frac{1}{2}$, the following identity
\begin{equation}\label{Self reciprocal}
\int_0^\infty K_\nu(t)\left \lbrace \cos(\pi \nu)\left( \frac{2}{\pi}K_{2\nu}(2\sqrt{xt})-Y_{2\nu}(2\sqrt{xt})  \right) - \sin(\pi \nu)J_{2\nu}(2\sqrt{xt}) \right\rbrace {\rm d}t = K_\nu(x)
\end{equation}
holds. It must be mentioned here that there are few functions in literature whose Koshliakov transforms have closed-form evaluations. However, it is always desirable to have such closed form evaluations, whenever possible, in view of their applications in number theory. Recently, such evaluations were obtained in \cite{DKesarwani}, \cite{DKumar} and \cite{DRoy} where they were fruitfully used to obtain interesting number-theoretic applications. 

In the same article \cite[Equation (15)]{Koshliakov2}, Koshliakov evaluated a general integral involving Koshliakov first kernel in terms of the generalized hypergeometric series. The result precisely states that for $\mu>-1/2$ and $\nu> -1/2+|\mu|$, the identity\footnote{It has been shown in \cite[Theorem 2.1]{DKK} that when $\mu\neq-\nu$, this result actually holds for $\nu\in\mathbb{C}\backslash\left(\mathbb{Z}\backslash\{0\}\right)$, $\textup{Re}(\mu)>-1/2$, $\textup{Re}(\nu)>-1/2$ and $\textup{Re}(\mu+\nu)>-1/2$; otherwise, it holds for $-1/2<\textup{Re}(\nu)<1/2$.}
\begin{align}\label{Koshliakov transform}
&\int_0^\infty K_{\mu}(t)t^{\mu+\nu} \left \lbrace \cos(\pi \nu)\left( \frac{2}{\pi}K_{2\nu}(2\sqrt{xt})-Y_{2\nu}(2\sqrt{xt})  \right) - \sin(\pi \nu)J_{2\nu}(2\sqrt{xt}) \right\rbrace {\rm d}t  \nonumber\\
&= \frac{\pi 2^{\mu+\nu-1}}{\sin(\nu \pi)} \left\lbrace \left(\frac{x}{2}\right)^{-\nu}\frac{\Gamma(\mu + \frac{1}{2})}{\Gamma(1-\nu)\Gamma(\frac{1}{2}-\nu)} {}_1F_2\left(\begin{matrix}
\mu + \frac{1}{2}\\ \frac{1}{2}- \nu, 1-\nu
\end{matrix} \bigg | \frac{x^2}{4}\right)\right. 
\left.- \left(\frac{x}{2}\right)^{\nu}\frac{\Gamma(\mu + \nu + \frac{1}{2})}{\Gamma(1+\nu)\Gamma(\frac{1}{2})} {}_1F_2\left(\begin{matrix}
\mu + \nu + \frac{1}{2}\\ \frac{1}{2}, 1+\nu
\end{matrix} \bigg | \frac{x^2}{4}\right)  \right\rbrace
\end{align}
holds. Equation \eqref{Self reciprocal} can be obtained as a special case of the above identity by plugging $\mu = -\nu$. 

Our generalization of the first Koshliakov kernel naturally leads us to study the integrals involving the general kernel $G_{\K, \nu}(xt)$. The following result provides the generalization of \eqref{Koshliakov transform}.
\begin{theorem}\label{Generalized Koshliakov transform}
For $\textup{Re}(\mu)$, $\textup{Re}(\nu)$, $\textup{Re}(\mu+\nu)>-1/2$, we have
\begin{equation}\label{Eqn:Generalized Koshliakov transform}
\int_0^\infty K_{\mu}(t)t^{\mu+\nu} G_{\K, \nu}(xt) \, {\rm d}t = 2^{\mu + \nu - 1} \MeijerG*{d+1}{1	}{1}{2d+1}{\frac{1-2\mu-\nu}{2}}{-\frac{\nu}{2}, (\frac{\nu}{2})_{r_1+r_2}, \left(\frac{1+\nu}{2}\right)_{r_2}, (\frac{\nu}{2})_{r_2}, \left(\frac{1+\nu}{2}\right)_{r_1+r_2}}{\frac{x^2}{4}}.
\end{equation}
In particular, for $-1/2<\Re(\nu)<1/2$, we have
\begin{equation}\label{Generalized self reciprocal}
\int_0^\infty K_{\nu}(t) G_{\K, \nu}(xt) \, {\rm d} t = \frac{1}{2} \MeijerG*{d+1}{0}{0}{\, 2d}{ - }{-\frac{\nu}{2}, (\frac{\nu}{2})_{r_1+r_2}, \left(\frac{1+\nu}{2}\right)_{r_2}, (\frac{\nu}{2})_{r_2}, \left(\frac{1+\nu}{2}\right)_{r_1+r_2-1}}{\frac{x^2}{4}}.
\end{equation}
\end{theorem}
\noindent
It is important to observe here that for $\K = \Q$,  \eqref{Koshliakov transform} can be obtained  as a special case of  \eqref{Eqn:Generalized Koshliakov transform} by applying Theorem \ref{Special case of new kernel} on the left hand side and Proposition \ref{Slater} (stated in the next section) on the right hand side of \eqref{Eqn:Generalized Koshliakov transform}.

In this article, we next study the Lambert series over any number field $\K$
\begin{equation*}
\sum_{n=1}^\infty \sigma_{\K, a}(n) e^{-ny} = \sum_{I \subseteq \mathcal{O}_\K} \frac{\left( \mathfrak{N}_{\K/\Q}(I)\right)^a}{\exp\left( \mathfrak{N}_{\K/\Q}(I) y\right) - 1}
\end{equation*}
where $a$ is any complex number. Recently, Dixit et. al.  \cite{DKK} studied the above  Lambert series for $\K = \Q$ and obtained an explicit transformation of the series for any complex number $a$. As a special case, their result provides the transformation formulas for Eisenstein series and Eichler integrals on SL$_2(\mathbb{Z})$, transfomration for the Dedekind eta function and Ramanujan's celebrated formula for $\zeta(2m+1)$. On the other hand, for $a$ even new transformation formulas have been obtained in \cite[Theorem 2.11, Corollary 2.13]{DKK}. In \cite{BK}, the authors have studied the same series for $\K$ being an imaginary quadratic field, which played a crucial role in studying the explicit identities for the Dedekind zeta function over an imaginary quadratic field.

The transformation of the above series for $\K = \Q$ has been established in \cite[Theorem 2.4]{DKK}, which precisely states that for $\textup{Re}(a)>-1$ and $\textup{Re}(y)>0$,
\begin{align*}
&\sum_{n=1}^\infty  \sigma_a(n)e^{-ny}+\frac{1}{2}\left(\left(\frac{2\pi}{y}\right)^{1+a}\mathrm{cosec}\left(\frac{\pi a}{2}\right)+1\right)\zeta(-a)-\frac{1}{y}\zeta(1-a)\nonumber\\
&=\frac{2\pi}{y\sin\left(\frac{\pi a}{2}\right)}\sum_{n=1}^\infty \sigma_{a}(n)\Bigg(\frac{(2\pi n)^{-a}}{\Gamma(1-a)} {}_1F_2\left(1;\frac{1-a}{2},1-\frac{a}{2};\frac{4\pi^4n^2}{y^2} \right) -\left(\frac{2\pi}{y}\right)^{a}\cosh\left(\frac{4\pi^2n}{y}\right)\Bigg).
\end{align*}
In the next theorem we obtain the generalization of the above result over any number field $\K$.
\begin{theorem}\label{General Lambert series transformation}
For $\Re(y) > 0$, the transformation
\begin{align}\label{Eqn:General Lambert series transformation}
\sum_{n=1}^\infty \sigma_{\K, a}(n) e^{-ny} = \frac{1}{y}\zeta_\K(1-a)- \frac{1}{2}\zeta_\K(-a) + \frac{\Gamma(a+1)\zeta(a+1)H}{y^{a+1}} + \frac{2^{1+\frac{a}{2}}\pi^{\frac{a+(1-a)d}{2}}D_{\mathbb{K}}^{\frac{a-1}{2}}}{y^{1+\frac{a}{2}}}\nonumber\\
\times\sum_{n=1}^\infty \sigma_{\mathbb{K},-a}(n)n^{\frac{a}{2}}\MeijerG*{d+1}{1}{1}{2d+1}{-\frac{a}{4}}{-\frac{a}{4},\left(\frac{a}{4}\right)_{r_1+r_2},\left(\frac{1}{2}+\frac{a}{4}\right)_{r_2};\left(\frac{a}{4}\right)_{r_2},\left(\frac{1}{2}+\frac{a}{4}\right)_{r_1+r_2}}{\frac{4\pi^{2(d+1)}n^2}{y^2D_{\mathbb{K}}^2}}.
\end{align}
holds where $a$ is any complex number with $\Re(a) > -1$.
\end{theorem}
The identity in Theorem \ref{General Lambert series transformation} can be continued analytically in the half plane $\Re(a) > -2m - 3$, where $m$ is any non-negative integer.
\begin{theorem}\label{Analytic continuation}
Let $m$ be a non-negative integer. Then for $\mathrm{Re}(a)>-2m-3$, we have
\begin{align}\label{continationeqn}
\sum_{n=1}^\infty \sigma_{\mathbb{K},a}(n)e^{-ny}&=\frac{1}{y}\zeta_\K(1-a)-\frac{1}{2}\zeta_{\mathbb{K}}(-a)+\frac{\Gamma(a+1)\zeta(a+1)H}{y^{a+1}}+\frac{2^{1+\frac{a}{2}}\pi^{\frac{a+(1-a)d}{2}}D_{\mathbb{K}}^{\frac{a-1}{2}}}{y^{1+\frac{a}{2}}}\nonumber\\
&\times\sum_{n=1}^\infty \sigma_{\mathbb{K},-a}(n)n^{\frac{a}{2}}\Bigg\{\MeijerG*{d+1}{1}{1}{2d+1}{-\frac{a}{4}}{-\frac{a}{4},\left(\frac{a}{4}\right)_{r_1+r_2},\left(\frac{1}{2}+\frac{a}{4}\right)_{r_2};\left(\frac{a}{4}\right)_{r_2},\left(\frac{1}{2}+\frac{a}{4}\right)_{r_1+r_2}}{\frac{4\pi^{2(d+1)}n^2}{y^2D_{\mathbb{K}}^2}}\nonumber\\
&-\frac{(-1)^{r_1}\pi^{\frac{d}{2}}2^{-(a+2)d}}{\sin\left(\frac{\pi a}{2} \right)^{r_1 +r_2}\cos\left(\frac{\pi a}{2} \right)^{r_2}} \left(\frac{2\pi^{d+1}n}{yD_\K} \right)^{-\frac{a}{2}-2} \sum_{k=0}^m\frac{(-1)^k}{\Gamma(-1-a-2k)} \left( \frac{(2\pi)^{d+1}ne^{-\frac{i\pi d}{2}}}{yD_\K} \right)^{-2k} \Bigg\}\nonumber\\
&+\frac{(-1)^{r_1}y(2\pi)^{-(a+2)d-1} \pi^{d-1}}{\sin\left(\frac{\pi a}{2} \right)^{r_1 +r_2}\cos\left(\frac{\pi a}{2} \right)^{r_2}}  \sum_{k=0}^m\frac{(-1)^k\zeta(2k+2)\zeta_{\mathbb{K}}(2k+a+2)}{\Gamma(-1-a-2k)} \left( \frac{(2\pi)^{d+1} e^{-\frac{i\pi d}{2}}}{yD_\K} \right)^{-2k}.
\end{align}
\end{theorem}
One can obtain \cite[Theorem 2.5]{DKK} for $\K= \Q$ and  \cite[Theorem 1.7]{BK} for $\K$ any imaginary quadratic field as a special case of the above theorem by mainly applying Proposition \ref{Slater}.

The paper is organised as follows. In \S \ref{Prelim}, we collect some results which are essential to prove our results. We prove Theorem \ref{zetazetathm} in \S \ref{Voro}. The section \S \ref{sp} is devoted to obtaining the special cases of our kernel over any number field. In \S \ref{Kosh}, we generalize the first Koshliakov transform over any number field in Theorem \ref{Generalized Koshliakov transform}. Finally, we establish the transformation formulas Theorem \ref{General Lambert series transformation} and Theorem \ref{Analytic continuation} for $\sigma_{\K, a}(n)$ in \S \ref{Trans}.

\section{Preliminaries}\label{Prelim}
In this section, we collect some basic ingredients, which we use throughout the paper. 

\subsection{Schwartz function}
A function is said to be a Schwartz function if all of its derivatives exist and decay faster than any polynomial. We denote the space of Schwartz functions on $\R$ by $\mathscr{S}(\R)$. For $f \in \mathscr{S}(\R)$, we let the Mellin transform of $f$ be $\mathcal{M}(f)$ i.e,
\begin{equation*}
\mathcal{M}(f)(s) = \int_0^\infty f(x)x^{s-1} dx.
\end{equation*}
The Mellin transform of any Schwartz function satisfies suitable analytic behaviour which has been given in \cite[Lemma 2.1]{BK}. For the sake of completeness we state here in the following lemma.
\begin{lemma}\label{Analyticity of Schwartz function}
The function $F(s)$ is absolutely convergent for $\Re(s) > 0$. It can be analytically continued to the whole complex plane except for simple poles at every non-positive integers. It also satisfies the functional equation:
\begin{equation*}\label{F functional equation}
\mathcal{M}(f')(s+1) = -s \, \mathcal{M}(f)(s), 
\end{equation*}
\end{lemma}
\subsection{Gamma function}
The Mellin transform of the Schwartz function $e^{-x}$ is known as Gamma function which can be defined for $\Re(s)>0$ via the convergent improper integral as 
\begin{equation}\label{gammadefn}
\Gamma(s) = \int_0^\infty e^{-x} x^{s-1} {\rm d}x.
\end{equation}
The analytic properties and functional equation of the $\Gamma$-function are given in the following proposition which follows immediately from the previous Lemma.
\begin{proposition}{\cite[Appendix A]{Ayoub}}
The integral in \eqref{gammadefn} is absolutely convergent for $\Re(s) > 0$. It can be analytically continued to the whole complex plane except for the simple poles at every non-positive integers. It also satisfies the functional equation:
\begin{equation*}\label{Gamma functional equation}
\Gamma(s+1) = s\Gamma(s).
\end{equation*}
\end{proposition}
The $\Gamma$-function satisfies many important properties. Here we mention two of them.
\begin{itemize}
\item[(i)]
 Euler's reflection formula : 
 \begin{equation}\label{Reflection formula}
 \Gamma(s)\Gamma(1-s) = \frac{\pi}{\sin \pi s}
 \end{equation}
 where $s \notin \mathbb{Z}$.
 \item[(ii)]
  Legendre's duplication formula :
  \begin{equation}\label{Duplication formula}
  \Gamma(s)\Gamma \left(s+\frac{1}{2}\right) = 2^{1-2s} \sqrt{\pi} \Gamma(2s).
  \end{equation}
\end{itemize}
Proofs of these properties can be found in \cite[Appendix A]{Ayoub}.

\subsection{Dedekind zeta function}\label{Dedekind}
Let $\K$ be any number field with extension degree $[\K : \Q] = d$ and signature $(r_1, r_2)$ (i. e., $d = r_1+2r_2$) and $D_\K$ denotes the absolute value of the discriminant of $\K$. Let $\mathcal{O}_\K$ be its ring of integers and $\mathfrak{N}$ be the norm map of $\K$ over $\Q$. Then the \begin{it}Dedekind zeta function\end{it} attached to number field $\K$ is defined by 
$$
\zeta_\K(s)=\sum_{\mathfrak{a}\subset\mathcal{O}_\K}\frac{1}{\mathfrak{N}(\mathfrak{a})^s}=\prod_{\mathfrak{p}\subset \mathcal{O}_\K}\bigg(1-\frac{1}{\mathfrak{N}(\mathfrak{p})^s}\bigg)^{-1},
$$
for all $s \in \C$ with $\mathfrak{R} (s)>1$, where $\mathfrak{a}$ and $\mathfrak{p}$ run over the non-zero integral ideals and prime ideals of $\mathcal{O}_\K$ respectively.
If $v_\K(m)$ denotes the number of non-zero integral ideals in $\mathcal{O}_\K$ with norm $m$, then $\zeta_\K$ can also be expressed as 
$$
\zeta_\K(s)=\sum_{m=1}^\infty \frac{v_\K(m)}{m^s}.
$$
The following proposition provides the analytic behaviour and the functional equation satisfied by the Dedekind zeta function.
\begin{proposition}\label{Analyticity of Dedekind zeta}
The function $\zeta_\K(s)$ is absolutely convergent for $\mathfrak{R}(s) > 1$. It can be analytically continued to the whole complex plane except for a simple pole at $s = 1$. It also satisfies the functional equation
\begin{equation}\label{functional equation}
\zeta_\K(s) = D_K^{\frac{1}{2}-s}2^{ds-r_2}\pi^{ds-r_1-r_2}\frac{\Gamma(1-s)^{r_1+r_2}}{\Gamma(s)^{r_2}}\sin\left(\frac{\pi s}{2}\right)^{r_1}\zeta_\K(1-s).
\end{equation}
\end{proposition}
\begin{proof}
For example, one can find a proof of this statement in \cite[pp. 254-255]{Lang}.
\end{proof}


\subsection{Special function}\label{sec:specialfunctions}
The mathematical functions which have more or less established names and notations due to their importance in mathematical analysis, functional analysis, geometry, physics, or other applications are known as special functions. These mainly appear as solutions of differential equations or integrals of elementary functions. 

One of the most important families of special functions are the Bessel functions.
The Bessel functions of the first kind and the second kind of order $\nu$ are defined by \cite[p.~40, 64]{watson-1944a}
\begin{align}\label{Bessel function}
	J_{\nu}(z)&:=\sum_{m=0}^{\infty}\frac{(-1)^m(z/2)^{2m+\nu}}{m!\Gamma(m+1+\nu)} \hspace{9mm} (z,\nu\in\mathbb{C}),\nonumber\\
	Y_{\nu}(z)&:=\frac{J_{\nu}(z)\cos(\pi \nu)-J_{-\nu}(z)}{\sin{\pi \nu}}\hspace{5mm}(z\in\mathbb{C}, \nu\notin\mathbb{Z}),
	\end{align}
	along with $Y_n(z)=\lim_{\nu\to n}Y_\nu(z)$ for $n\in\mathbb{Z}$. 
The modified Bessel functions of the first and second kinds are defined by \cite[p.~77, 78]{watson-1944a}
\begin{align}
I_{\nu}(z)&:=
\begin{cases}
e^{-\frac{1}{2}\pi\nu i}J_{\nu}(e^{\frac{1}{2}\pi i}z), & \text{if $-\pi<\arg(z)\leq\frac{\pi}{2}$,}\nonumber\\
e^{\frac{3}{2}\pi\nu i}J_{\nu}(e^{-\frac{3}{2}\pi i}z), & \text{if $\frac{\pi}{2}<\arg(z)\leq \pi$,}
\end{cases}\nonumber\\
K_{\nu}(z)&:=\frac{\pi}{2}\frac{I_{-\nu}(z)-I_{\nu}(z)}{\sin\nu\pi}\label{kbesse}
\end{align}
respectively. When $\nu\in\mathbb{Z}$, $K_{\nu}(z)$ is interpreted as a limit of the right-hand side of \eqref{kbesse}. 

The generalized hypergeometric function is defined by the following power series :
\begin{equation*}
\pFq{p}{q}{a_1, a_2, \cdots, a_p}{b_1,b_2, \cdots, b_q}{z}:=\sum_{n=0}^{\infty}\frac{(a_1)_n(a_2)_n\cdots(a_p)_n}{(b_1)_n(b_2)_n\cdots(b_q)_n}\frac{z^n}{n!}
\end{equation*}
where $(a)_n$ denotes the Pochhammer symbol defined by $(a)_n:=a(a+1)\cdots(a+n-1)=\G(a+n)/\G(a)$.
It is well-known \cite[p.~62, Theorem 2.1.1]{AAR} that the above series converges absolutely for all $z$ if $p\leq q$ and for $|z|<1$ if $p=q+1$, and it diverges for all $z\neq0$ if $p>q+1$ and the series does not terminate.

The G-function was introduced initially by Meijer as a very general function using a series. Later, it was defined more generally via a line integral in the complex plane (cf. \cite{erd1}) given by 
\begin{equation}\label{G-function}
\begin{aligned}
G^{m, \ n}_{p, \ q}\bigg(\begin{matrix}
a_1, \ldots, a_p \\
b_1, \ldots, b_q
\end{matrix} \ \bigg|\ z\bigg)=\frac{1}{2\pi i}\underset{{(C)}}{\bigints} \frac{\prod\limits_{j=1}^m\Gamma(b_j-s)\prod\limits_{j=1}^n\Gamma(1-a_j+s)}{\prod\limits_{j=m+1}^q\Gamma(1-b_j+s)\prod\limits_{j=n+1}^p\Gamma(a_j-s)}z^s \rm{d}s ,
\end{aligned}
\end{equation}
where $z \neq 0$ and $m$, $n$, $p$, $q$ are integers which satisfy $0 \leq m \leq q$ and $0 \leq n \leq p$. The poles of the integrand must be all simple. Here $(C)$ in the integral denotes the vertical line from $C-i\infty$ to $C+i\infty$  such that all poles of $\Gamma(b_j-s)$ for $ j=1, \ldots, m$, must lie on one side of the vertical line while all poles of $\Gamma(1-a_j+s)$ for $ j=1, \ldots, n$ must lie on the other side. The integral then converges for $|\arg z| < \delta \pi$ where
$$\delta = m+n - \frac{1}{2}(p+q).$$
The integral additionally converges for $|\arg z|= \delta \pi$ if $(q-p)(\Re(s) + 1/2) > \Re (v) + 1$, where 
$$
v = \sum_{j=1}^{q}b_j - \sum_{j=1}^{p}a_j.
$$
Special cases of the $G$-function include many other special functions. 
The following proposition \cite[Formula 6.5.1, p. 230]{Luke} provides the relation between the $G$-function and the generalized hypergeometric function.
\begin{proposition}\label{Slater}
For $p < q$ or $p=q$ with $|z|<1$, we have
\begin{align*}
G^{m, \ n}_{p, \ q}\bigg(\begin{matrix}
 {\bf a}_p \\
{\bf b}_q
\end{matrix} \ \bigg|\ z\bigg)= \sum_{h=1}^m \frac{\prod_{j=1}^m \Gamma(b_j - b_h)^* \prod_{j=1}^n \Gamma(1+ b_h- a_j) z^{b_h}}{\prod_{j=m+1}^q \Gamma(1+ b_h - b_j) \prod_{j=n+1}^p \Gamma(a_j - b_h)} \pFq{p}{q-1}{1+b_h-{\bf a}_p}{(1+b_h-{\bf b}_q)^*}{z}
\end{align*}
where $ {\bf a}_p = (a_1, \ldots ,a_p) ,  {\bf b}_q = (b_1, \ldots ,b_q) $ and for some fixed $h$, $(b_j - b_h)^* = (b_1 - b_h, \ldots ,b_{h-1}-b_h, b_{h+1}-b_h, \ldots ,b_q-b_h)$.
\end{proposition}

\section{Vorono\"{\dotlessi} summation formula for $\sigma_{\K,a}(n)$}\label{Voro}
In this section, we prove the Vorono\"{\dotlessi} type identity for $\sigma_{\K,a}(n)$, where $\K$ is an arbitrary number field. Throughout the paper, let $\int_{(c)}$ denotes the line integral $\int_{c-i\infty}^{c+i\infty}$.
\subsection{Proof of Theorem \ref{zetazetathm}}
We begin by showing that the series on the right-hand side of \eqref{zetazetathmeqn} converges for $-\frac{1}{2}<\Re(a)<\frac{1}{2}$ and for that we need the bounds for $\sigma_{\K, -a}(n)$ and of our kernel. The result in \cite[Corollary 7.119, p. 430]{Bordelles} implies that 
$\sigma_{\K, -a}(n) \leq \sum_{d \mid n} \sigma_0(d) d^{-a},$ where $\sigma_0(d)$ denotes the divisor function $d(n)$. Employing the elementary bound of divisor function we can bound $\sigma_{\K, -a}(n)$ as
\begin{align}\label{Bound for sigma}
\sigma_{\K, -a}(n) \ll \begin{cases}
n^\epsilon & \text{ for } \Re(a)>0\\
n^{\epsilon - \Re(a)} & \text{ for } \Re(a)<0
\end{cases}
\end{align}
where $\epsilon>0$ is arbitrarily small. We next establish the bound for the kernel. It follows from \cite[Equation (4), p. 191]{Luke} that for  $\sigma=q-p$, $\nu=q-m$, $\theta=\frac{1}{\sigma}\left\{\frac{1}{2}(1-\sigma)+\sum_{h=1}^qb_h\right\}$ and $A_{q}^{m, 0}=(-1)^\nu(2\pi i)^{-\nu}\exp\left(- i\pi\sum_{j=m+1}^qb_j\right)$, as $z\to\infty$,
\begin{align*}
\MeijerG*{m}{0}{p}{q}{a_1,\cdots,a_m}{b_1,\cdots,b_q}{z}\sim A_{q}^{m,0}H_{p,q}\left(ze^{i\pi(q-m)}\right)+\overline{A}_{q}^{m,0}H_{p,q}\left(ze^{-i\pi(q-m)}\right),
\end{align*}
where  (cf. \cite[Equation (13), p. 180]{Luke})
\begin{align*}
H_{p,q}(z):=\frac{(2\pi)^{\frac{\sigma-1}{2}}}{\sqrt{\sigma}}e^{-\sigma z^{1/\sigma}}z^\theta+O\left(z^{\theta-1/\sigma}e^{-\sigma z^{1/\sigma}}\right).
\end{align*}
Here $\overline{A}_{q}^{m,0}$ denotes $A_{q}^{m,0}$ with $i$ replaced by $-i$.
Invoking the above formula for $m=d+1,\ p=0$ and $\ q=2(d+1)$, it follows from the definition of Meijer G-function \eqref{G-function} that as $t\to\infty$,
\begin{align}\label{bound for kernel}
G_{\mathbb{K},\frac{a}{2}}\left(\frac{4\pi^{d+1}nt}{D_{\mathbb{K}}}\right)
&=\frac{(-1)^{d+1}(nt)^\frac{ad-a-d}{2(d+1)}}{\sqrt{2\pi(2d+2)}}\left\{\exp\left(\frac{i\pi(-d-1)}{2}+\frac{i\pi(ad-a-d)}{4}-2i(d+1)(nt)^{\frac{1}{(d+1)}}
\right.\right.\nonumber\\
&-\left.\left.\frac{i\pi}{4}(2(1+r_1+r_2)+a(d-1))\right)
+\exp\left(\frac{-i\pi(-d-1)}{2}-\frac{i\pi(ad-a-d)}{4}\right.\right.
\nonumber\\
&+\left.\left.2i(d+1)(nt)^{\frac{1}{(d+1)}}+\frac{i\pi}{4}(2(1+r_1+r_2)+a(d-1))\right)\right\}+O\left((n^2t^2)^{\frac{ad-a-d}{4(d+1)}-\frac{1}{2(d+1)}}\right)\nonumber\\
&=\frac{2(-1)^{d+1}(nt)^\frac{ad-a-d}{2(d+1)}}{\sqrt{2\pi(2d+2)}}\cos\left(-\frac{\pi}{2}(d+1)+\frac{\pi}{4}(ad-a-d)-2(d+1)(nt)^{\frac{1}{(d+1)}}\right.
\nonumber\\
&-\left.\frac{\pi}{4}(2(1+r_1+r_2)+a(d-1))\right)+O\left((n^2t^2)^{\frac{ad-a-d}{4(d+1)}-\frac{1}{2(d+1)}}\right)\nonumber\\
&=\frac{2(-1)^{d+1}(nt)^\frac{ad-a-d}{2(d+1)}}{\sqrt{2\pi(2d+2)}}\left\{A_{a,d}\cos\left(2(d+1)(nt)^{\frac{1}{(d+1)}}\right)+B_{a,d}\sin\left(2(d+1)(nt)^{\frac{1}{(d+1)}}\right)\right\}\nonumber\\
&+O\left((n^2t^2)^{\frac{ad-a-d}{4(d+1)}-\frac{1}{2(d+1)}}\right),
\end{align}
where in the last step $A_{a,d}$ and $B_{a,d}$ denote the constants depending only on $a$ and $d$. 

The integral on the right hand side of \eqref{zetazetathmeqn} inside the sum can be split into two different parts as
\begin{align*}
\int_0^\infty t^{\frac{a}{2}}G_{\mathbb{K},\frac{a}{2}}\left(\frac{4\pi^{d+1}nt}{D_{\mathbb{K}}}\right) f(t)\ dt&=\left(\int_0^M+\int_M^\infty\right)t^{\frac{a}{2}}G_{\mathbb{K},\frac{a}{2}}\left(\frac{4\pi^{d+1}nt}{D_{\mathbb{K}}}\right) f(t)\ dt
\end{align*}
for sufficiently large $M$. We denote the first and the second integral of the right hand side of the above equation by $I_1(n,a)$ and $I_2(n,a)$ respectively. We first show the convergence of $\sum_{n=1}^\infty\sigma_{\mathbb{K},-a}(n)n^{\frac{a}{2}}I_2(n,a)$. For $f$ being a Schwatrz function, we have $f(t) \ll t^{-\alpha}$ for $\alpha$ large enough. Thus the bound \eqref{bound for kernel} implies
\begin{align*}
I_2(n,a)\ll n^{\frac{ad-a-d}{2(d+1)}}\int_M^\infty t^{\frac{\mathrm{Re}(a)}{2}+\mathrm{Re}\left(\frac{ad-a-d}{2(d+1)}\right)-\alpha}\left\{A_{a,d}\cos\left(2(d+1)(nt)^{\frac{1}{(d+1)}}\right)+B_{a,d}\sin\left(2(d+1)(nt)^{\frac{1}{(d+1)}}\right)\right\}dt.
\end{align*}
Here the two integrals involving sine and cosine function behave similarly. Therefore, it is enough to prove the convergence of the integral involving cosine function. The change of variable $t^{\frac{1}{d+1}}=x$ yields
\begin{align*}
\int_M^\infty t^{\frac{\mathrm{Re}(a)}{2}+\mathrm{Re}\left(\frac{ad-a-d}{2(d+1)}\right)-\alpha}\cos\left(2(d+1)(nt)^{\frac{1}{(d+1)}}\right)dt&=(d+1)\int_{M^{\frac{1}{d+1}}}^\infty x^{\frac{2ad+d}{2}-(d+1)\alpha}\cos\left(2(d+1)n^{\frac{1}{(d+1)}}x\right)dt.
\end{align*} 
Performing the integration by parts on the integral on the right-hand side of the above equation, we obtain
\begin{align*}
&\int_M^\infty t^{\frac{\mathrm{Re}(a)}{2}+\mathrm{Re}\left(\frac{ad-a-d}{2(d+1)}\right)-\alpha}\cos\left(2(d+1)(nt)^{\frac{1}{(d+1)}}\right)dt\nonumber\\
&=(d+1)\left\{\frac{1}{n^{\frac{1}{d+1}}}\frac{(M^{1/(d+1)})^{\frac{2ad+d}{2}-(d+1)\alpha}\sin\left(2(d+1)(nM)^{\frac{1}{(d+1)}}\right)}{2(d+1)}\right.\nonumber\\
&\hspace{1.7cm}\left.-\frac{1}{n^{\frac{1}{d+1}}}\left(\frac{2ad+d}{2}-(d+1)\alpha\right)\int_0^\infty x^{\frac{2ad+d-2}{2}-(d+1)\alpha}\sin\left(2(d+1)n^{\frac{1}{(d+1)}}x\right)dt\right\}\nonumber\\
&\ll\frac{\sin\left(2(d+1)(nM)^{\frac{1}{(d+1)}}\right)}{n^{\frac{1}{d+1}}},
\end{align*}
as $n\to\infty$.
We now apply the above bounds to conclude
\begin{align*}
\sum_{n=1}^\infty\sigma_{\mathbb{K},-a}(n)n^{\frac{a}{2}}I_2(n,a)&\ll\sum_{n=1}^\infty \sigma_{\mathbb{K},-a}(n)n^{\frac{\mathrm{Re}(a)}{2}+\mathrm{Re}\left(\frac{ad-a-d}{2(d+1)}\right)-\frac{1}{d+1}}\sin\left(2(d+1)(nM)^{\frac{1}{(d+1)}}\right).
\end{align*}
It follows from the bound \eqref{Bound for sigma} of $\sigma_{\K, -a}(n)$ that
\begin{align*}
\sum_{n=1}^\infty\sigma_{\mathbb{K},-a}(n)n^{\frac{a}{2}}I_2(n,a)&\ll\sum_{n=1}^\infty \frac{\sin\left(2(d+1)(nM)^{\frac{1}{(d+1)}}\right)}{n^{\frac{1}{d+1}-\frac{\mathrm{Re}(a)}{2}-\mathrm{Re}\left(\frac{ad-a-d}{2(d+1)}\right)-\epsilon}} \hspace{2cm} \text{ for } \Re(a)\geq 0,
\end{align*}
and
\begin{align*}
\sum_{n=1}^\infty\sigma_{\mathbb{K},-a}(n)n^{\frac{a}{2}}I_2(n,a)&\ll\sum_{n=1}^\infty \frac{\sin\left(2(d+1)(nM)^{\frac{1}{(d+1)}}\right)}{n^{\frac{1}{d+1}+\frac{\mathrm{Re}(a)}{2}-\mathrm{Re}\left(\frac{ad-a-d}{2(d+1)}\right)-\epsilon}} \hspace{2cm} \text{ for } \Re(a)< 0.
\end{align*}
The Dirichlet's test for the convergence of the series implies that the series $\sum_{n=1}^\infty\sigma_{\mathbb{K},-a}(n)n^{\frac{a}{2}}I_2(n,a)$ is convergent for $-1-\frac{d}{2}<\mathrm{Re}(a)<\frac{1}{2}$. Therefore, this series is convergent for $-\frac{1}{2}<\mathrm{Re}(a)<\frac{1}{2}$.

The similar argument as above also shows that the series $\sum_{n=1}^\infty\sigma_{\mathbb{K},-a}(n)n^{\frac{a}{2}}I_1(n,a)$ is convergent for $-\frac{1}{2}<\mathrm{Re}(a)<\frac{1}{2}$ as the bound in \eqref{bound for kernel} is also valid for $n\to\infty$. This proves the fact that the series in \eqref{zetazetathmeqn} is convergent for $-\frac{1}{2}<\mathrm{Re}(a)<\frac{1}{2}$.

For $f \in \mathscr{S}(\R)$ and $ \Re(s) := c > \max (1, 1+\Re(a))$, the inverse Mellin transform of $F$ yields 
\begin{equation}\label{Eqn:Voronoi sum}
I_{\K, a} = \sum_{n=1}^{\infty} \sigma_{\K, a}(n) f(n) = \sum_{n=1}^{\infty} \sigma_{\K, a}(n) \frac{1}{2\pi i} \int_{(c)} F(s) n^{-s} {\rm d}s = \frac{1}{2\pi i}\int_{(c)}F(s)\zeta(s)\zeta_{\K}(s-a) {\rm d}s, 
\end{equation}
where in the last step, we have written the Dirichlet series associated to the divisor function $\sigma_{\K, a}(n)$ as
\begin{align}\label{Dirichlet series of zeta zetak}
\sum_{n=1}^\infty \frac{\sigma_{\K, a}(n)}{n^s} = \zeta(s)\zeta_{\K}(s-a) \hspace{3cm}  (\Re(s)>1 \text{ and } \Re(s-a)>1).
\end{align}
We next consider the contour $\mathscr{C}$ given by the rectangle with vertices $\{c - iT,c + iT, \lambda + iT, \lambda - iT\}$ in the anticlockwise direction for sufficiently large $T$ where $-1<\lambda<0$. It follows from Lemma \ref{Analyticity of Schwartz function}, the analytic behaviour of $\zeta(s)$ and Proposition \ref{Analyticity of Dedekind zeta} that the integrand is analytic inside the contour except for the possible simple poles at $s = 0, 1$ and $1+a$. Invoking the Cauchy residue theorem, we have
\begin{equation}\label{Eqn:Cauchy theorem}
\frac{1}{2\pi i}\int_{\mathscr{C}}F(s)\zeta(s)\zeta_{\K}(s-a) {\rm d}s = \mathcal{R}_0 + \mathcal{R}_1 + \mathcal{R}_{1+a}
\end{equation}
where  $\mathcal{R}_{z_0}$ denotes the residue of the integrand at $z_0$. We next evaluate the values of $\mathcal{R}_0$, $\mathcal{R}_1 $ and $\mathcal{R}_{1+a}$ using Lemma \ref{Analyticity of Schwartz function} and Proposition \ref{Analyticity of Dedekind zeta}, which are given by
$$\mathcal{R}_0 = \lim_{s \to 0} s F(s)\zeta(s)\zeta_{\K}(s-a) = \frac{1}{2} \mathcal{M}(f')(1) \zeta_{\K}(-a) = \frac{\zeta_{\K}(-a)}{2} \int_0^\infty f'(t) \, {\rm d}t =- \frac{\zeta_{K}(-a) f(0^+)}{2},$$
$$\mathcal{R}_1 = \lim_{s \to 1} (s-1) F(s)\zeta(s)\zeta_{\K}(s-a) = F(1)\zeta_\K(1-a) = \zeta_\K(1-a) \int_0^\infty f(t) \, {\rm d}t$$
and
$$\mathcal{R}_{1+a} = \lim_{s \to 1+a} (s-1-a) F(s)\zeta(s)\zeta_{\K}(s-a) = F(1+a)\zeta(1+a)H= H \zeta(1+a) \int_0^\infty f(t) t^a \, {\rm d}t$$
Substituting the values of $\mathcal{R}_0$, $\mathcal{R}_1 $ and $\mathcal{R}_{1+a}$ in \eqref{Eqn:Cauchy theorem}, the equations \eqref{Eqn:Voronoi sum} and \eqref{Eqn:Cauchy theorem} together yield
\begin{equation}\label{Hori and Vert Int}
I_{\K, a} = \int_0^\infty \left(\zeta_\K(1-a)+ t^a \frac{2\pi h \zeta(1+a)}{w\sqrt{D_\K}} \right)  f(t) \, {\rm d}t -\frac{1}{2} \zeta_\K(-a)f(0^+) + \mathcal{H}_1 + \mathcal{H}_2 + \mathcal{V}
\end{equation}
where $\mathcal{H}_1 := \lim\limits_{T \to \infty} \frac{1}{2\pi i}\int_{\lambda + iT}^{c+iT} F(s)\zeta(s)\zeta_{\K}(s-a) \, {\rm d}s$ and $\mathcal{H}_2 := \lim\limits_{T \to \infty} \frac{1}{2\pi i}\int_{c - iT}^{\lambda - iT} F(s)\zeta(s)\zeta_{\K}(s-a) \, {\rm d}s$ are the horizontal integrals and $\mathcal{V} := \frac{1}{2\pi i}\int_{(\lambda)} F(s)\zeta(s)\zeta_{\K}(s-a) \, {\rm d}s$ is the vertical integral.

It follows from a standard argument of the Phragmen-Lindel{\"o}f principle [cf. \cite[Chapter 5]{Iwaniec}] and the functional equation of both zeta functions that for $s = \sigma + it$ with $\lambda < \sigma < c$ and for some $\theta \in \R$,
$$ |\zeta(\sigma+it)\zeta_{\K}(\sigma+it)| \ll t^{\theta (1-\sigma)},\quad \mathrm{as}\ t\to\infty.$$
On the other hand according to our hypothesis,  $F(s)$ decays faster than any polynomial in $t$ in the above vertical strip. Thus, the horizontal integrals $\mathcal{H}_1$ and $\mathcal{H}_2$ vanish. We next concentrate in evaluating the vertical integral $\mathcal{V}$. 
It follows from Proposition \ref{Analyticity of Dedekind zeta} that
\begin{align}\label{functprod}
\zeta(s)\zeta_{\K}(s-a)&=D_\K^{\frac{1}{2}-s+a}2^{(d+1)s-da-r_2}\pi^{(d+1)s-da-r_1-r_2-1}\frac{\Gamma(1-s)\Gamma(1-s+a)^{r_1+r_2}}{\Gamma(s-a)^{r_2}}\sin\left(\frac{\pi s}{2}\right)\nonumber\\
&\qquad\times\sin\left(\frac{\pi (s-a)}{2}\right)^{r_1}\zeta(1-s)\zeta_\K(1-s+a).
\end{align}
Therefore \eqref{functprod} reduces the vertical integral as
\begin{align*}
\mathcal{V}&=D_\K^{a+\frac{1}{2}}2^{-da-r_2}\pi^{-da-r_1-r_2-1}\frac{1}{2\pi i}\int_{(\lambda)}F(s)\frac{\Gamma(1-s)\Gamma(1-s+a)^{r_1+r_2}}{\Gamma(s-a)^{r_2}}\sin\left(\frac{\pi s}{2}\right)\sin\left(\frac{\pi (s-a)}{2}\right)^{r_1}\nonumber\\
&\qquad\qquad\qquad\qquad\qquad\qquad\qquad\qquad\times\zeta(1-s)\zeta_\K(1-s+a)\left(\frac{(2\pi)^{d+1}}{D_\K}\right)^s\ ds\nonumber\\
&=D_\K^{a-\frac{1}{2}}2^{d(1-a)-r_2+1}\pi^{d(1-a)-r_1-r_2}\frac{1}{2\pi i}\int_{(1-\lambda)}\frac{F(1-s)\Gamma(s)\Gamma(s+a)^{r_1+r_2}}{\Gamma(1-s-a)^{r_2}}\sin\left(\frac{\pi }{2}(1-s)\right)\nonumber\\
&\qquad\qquad\qquad\qquad\qquad\times\sin\left(\frac{\pi }{2}(1-s-a)\right)^{r_1}\zeta(s)\zeta_\K(s+a)\left(\frac{(2\pi)^{d+1}}{D_\K}\right)^{-s}\ ds,
\end{align*}
where in the last step we change the variable $s$ by $1-s$ in the integral. Now substituting $s$ by $s-a$ and letting $\lambda^* = 1-\lambda+\Re(a)$ in the above equation, our vertical integral $\mathcal{V}$ becomes
\begin{multline}\label{z1}
\mathcal{V}
=\frac{2(2\pi)^{a+d-r_2}}{\pi^{r_1}\sqrt{D_\K}}\frac{1}{2\pi i}\int_{(\lambda^*)}\frac{F(1+a-s)\Gamma(s-a)\Gamma(s)^{r_1+r_2}}{\Gamma(1-s)^{r_2}}\cos\left(\frac{\pi }{2}(s-a)\right)
\cos\left(\frac{\pi s}{2}\right)^{r_1}\\
\times\zeta(s-a)\zeta_\K(s)\left(\frac{(2\pi)^{d+1}}{D_\K}\right)^{-s}\ ds.
\end{multline}
For Re$(s)>1$ and Re$(s-a)>1$, it follows that
\begin{align*}
\zeta(s-a)\zeta_\K(s)=\sum_{n=1}^\infty\frac{\sigma_{\K,-a}(n)}{n^{s-a}}.
\end{align*}
Therefore, the integral \eqref{z1} can be written as
\begin{align}\label{Ix1}
\mathcal{V}
&=\frac{2(2\pi)^{a+d-r_2}}{\pi^{r_1}\sqrt{D_\K}} \sum_{n=1}^\infty \frac{\sigma_{\K,-a}(n)}{n^{-a}}I_{\K,a}(n)
\end{align}
where
\begin{align}\label{ikan}
I_{\K,a}(n):=\frac{1}{2\pi i}\int_{(\lambda^*)}F(1+a-s)N_{\K,a}(s)\left(\frac{(2\pi)^{d+1}n}{D_\K}\right)^{-s}\ ds,
\end{align}
and
\begin{align*}
N_{\K,a}(s):=\frac{\Gamma(s-a)\Gamma(s)^{r_1+r_2}}{\Gamma(1-s)^{r_2}}\cos\left(\frac{\pi }{2}(s-a)\right)\cos\left(\frac{\pi s}{2}\right)^{r_1}.
\end{align*}
We apply \eqref{Reflection formula} and \eqref{Duplication formula} together to reduce the above factor $N_{\K,a}(s)$ as
\begin{align}\label{nka2s}
N_{\K,a}(s) &=2^{(1+d)s-a-1-(r_1+r_2)}\pi^{\frac{r_1}{2}+\frac{1}{2}}\frac{\Gamma\left(\frac{s}{2}-\frac{a}{2}\right)\Gamma\left(\frac{s}{2}\right)^{r_1+r_2}\Gamma\left(\frac{s}{2}+\frac{1}{2}\right)^{r_2}}{\Gamma\left(\frac{1}{2}-\frac{s}{2}\right)^{r_1+r_2}\Gamma\left(1-\frac{s}{2}\right)^{r_2}\Gamma\left(\frac{1}{2}+\frac{a}{2}-\frac{s}{2}\right)}.
\end{align}
On the other hand, applying Proposition \ref{Analyticity of Schwartz function} into \eqref{ikan}, $I_{\K,a}(n)$ can be dictated as
\begin{align}\label{beforeintbyparts}
I_{\K,a}(n)&=-\frac{1}{2\pi i}\int_{(\lambda^*)}\int_0^\infty\frac{N_{\K,a}(s)f'(t)t^{1+a-s}}{1+a-s}\left(\frac{(2\pi)^{d+1}n}{D_\K}\right)^{-s}\ dt ds\nonumber\\
&=-\frac{1}{n^{1+a}}\int_0^\infty f'(t)\left(\frac{1}{2\pi i}\int_{(\lambda^*)}\frac{N_{\K,a}(s)(nt)^{1+a-s}}{1+a-s}\left(\frac{(2\pi)^{d+1}}{D_\K}\right)^{-s}\ ds\right)\ dt\nonumber\\
&=-\frac{1}{n^{1+a}}\int_0^\infty f'(t)J_{\K,a}(nt)\ dt,
\end{align}
where
\begin{align*}
J_{\K,a}(x):=\frac{1}{2\pi i}\int_{(\lambda^*)}\frac{N_{\K,a}(s)x^{1+a-s}}{1+a-s}\left(\frac{(2\pi)^{d+1}}{D_\K}\right)^{-s}\ ds.
\end{align*}
Thus the integration by parts on the integral \eqref{beforeintbyparts} provides
\begin{align}\label{afterintbyparts}
I_{\K,a}(n)&=\frac{1}{n^{a+1}}\int_0^\infty f(t)\frac{d}{dt}\left( J_{\K,a}(nt)\right)\ dt.
\end{align}
Differentiating $J_{\K,a}(nt)$ with respect to $t$, we get
\begin{align}\label{beforenka2s}
\frac{d}{dt}\left(J_{\K,a}(nt)\right)&=\frac{n^{a+1}t^a}{2\pi i}\int_{(\lambda^*)}N_{\K,a}(s)\left(\frac{(2\pi)^{d+1}nt}{D_\K}\right)^{-s}\ ds.
\end{align}
We next insert the factor $N_{\K, a}(s)$ from \eqref{nka2s} and replace $s$ by $\frac{a}{2}-2s$ into \eqref{beforenka2s} to deduce
\begin{align}\label{ddny}
\frac{d}{dt}\left(J_{\K,a}(nt)\right) &=\frac{n^{\frac{a}{2}+1}(tD_\K)^{\frac{a}{2}}\pi^{\frac{r_1+1- (d+1)a}{2}}}{2^{a+r_1+r_2}} \frac{1}{2\pi i}\int_{\left(-\frac{\lambda^*}{2} + \frac{a}{4}\right)}\frac{\Gamma\left(-\frac{a}{4}-s\right)\Gamma\left(\frac{a}{4}-s \right)^{r_1+r_2}\Gamma\left(\frac{1}{2} +\frac{a}{4} -s \right)^{r_2}}{\Gamma\left(\frac{1}{2} - \frac{a}{4}+s \right)^{r_1+r_2}\Gamma(1 - \frac{a}{4}+s)^{r_2}\Gamma\left(\frac{1}{2}+\frac{a}{4}+s\right)}\nonumber\\
 &\hspace{10.5cm} \times \left(\frac{\pi^{2(d+1)}n^2t^2}{D_\K^2}\right)^{s}\ ds\nonumber\\
&=\frac{n^{\frac{a}{2}+1}(tD_\K)^{\frac{a}{2}}\pi^{\frac{r_1+1- (d+1)a}{2}}}{2^{a+r_1+r_2}} G_{\K, a/2}\left(\frac{4\pi^{d+1}nt}{D_\K} \right),
\end{align}
where in the last step, we apply the definitions of Meijer G-function \eqref{G-function} and the kernel \eqref{Our Kernel}. Invoking \eqref{ddny} into \eqref{afterintbyparts} and inserting the resulting expression into \eqref{Ix1}, we evaluate the vertical integral as
\begin{align}\label{bsfinal}
\mathcal{V} = 2 \pi^{ \frac{1+a +(1-a)d}{2}}D_{\K}^{\frac{a-1}{2}}\sum_{n=1}^\infty \sigma_{\K,-a}(n)n^{a/2} \int_0^\infty t^{a/2}G_{\K, a/2}\left(\frac{4\pi^{(d+1)}nt}{D_\K}\right)f(t)\ {\rm d}t.
\end{align}
We finally substitute \eqref{bsfinal} into \eqref{Hori and Vert Int} to arrive at \eqref{zetazetathmeqn} which concludes our theorem.

\section{Special cases of the generalized kernel}\label{sp}
In this section, we study the special cases of our generalized kernel. For instance, we have shown here that for $\K = \Q$ i.e, for the extension degree $d=1$, our kernel reduces to the first Koshliakov kernel \eqref{First Koshliakov kernel}. 
\subsection{Proof of Theorem \ref{Special case of new kernel}}
It follows from the definition \eqref{Our Kernel} that for $\K = \Q$, our kernel takes the form
\begin{equation*}
G_{\Q, \nu}(xt) = \MeijerG*{2}{0}{0}{4}{-}{-\frac{\nu}{2},\frac{\nu}{2};\frac{1-\nu}{2},\frac{1+\nu}{2}}{\frac{x^2t^2}{16}}
\end{equation*}
We invoke Proposition \ref{Slater} on the right hand side of the above equation to write the kernel as
\begin{align}\label{g*}
G_{\Q, \nu}(xt) &=\frac{\Gamma(\nu)}{\Gamma\left(\frac{1}{2}\right)\Gamma\left(\frac{1}{2}-\nu\right)}\left(\frac{x^2t^2}{16}\right)^{-\frac{\nu}{2}}\pFq03{-}{1-\nu,\frac{1}{2}-\nu,\frac{1}{2}}{\frac{x^2t^2}{16}}\nonumber\\
&+\frac{\Gamma(-\nu)}{\Gamma\left(\frac{1}{2}\right)\Gamma\left(\frac{1}{2}+\nu\right)}\left(\frac{x^2t^2}{16}\right)^{\frac{\nu}{2}}\pFq03{-}{1+\nu,\frac{1}{2}+\nu,\frac{1}{2}}{\frac{x^2t^2}{16}}.
\end{align}
The series definition of ${}_0F_3$ yields
\begin{align}\label{2of3}
\pFq03{-}{1 \pm \nu,\frac{1}{2}\pm\nu,\frac{1}{2}}{\frac{x^2t^2}{16}}&=\frac{1}{2}\left(\pFq01{-}{1\pm2\nu}{xt}+\pFq01{-}{1\pm2\nu}{-xt}\right).
\end{align}
Inserting \eqref{2of3} into \eqref{g*} and applying \eqref{Duplication formula} and \eqref{Reflection formula} on the Gamma factors,  \eqref{g*} can be simplified as
\begin{align}\label{78}
G_{\Q, \nu}(xt) &=\frac{1}{2\sin(\pi\nu)}\Bigg\{\frac{(xt)^{-\nu}}{\Gamma(1-2\nu)}\pFq01{-}{1-2\nu}{xt}-\frac{(xt)^{\nu}}{\Gamma(1+2\nu)}\pFq01{-}{1+2\nu}{xt}\nonumber\\
&\quad+\frac{(xt)^{-\nu}}{\Gamma(1-2\nu)}\pFq01{-}{1-2\nu}{-xt}-\frac{(xt)^{\nu}}{\Gamma(1+2\nu)}\pFq01{-}{1+2\nu}{-xt}\Bigg\}.
\end{align}
The $K$-Bessel function and the $J$-Bessel function satisfies the following relations with the hypergeometric function which are
\begin{align*}
K_{\nu}(x)=&\frac{\pi}{2\sin(\nu\pi)}\left\{\frac{(x/2)^{-\nu}}{\G(1-\nu)}\pFq01{-}{1-\nu}{\frac{x^2}{4}}-\frac{(x/2)^{\nu}}{\G(1+\nu)}\pFq01{-}{1+\nu}{\frac{x^2}{4}}\right\}.
\end{align*}
and
\begin{align*}
J_\nu(x)=\frac{(z/2)^\nu}{\Gamma(\nu+1)}\pFq01{-}{\nu+1}{-x^2/4}.
\end{align*}
Thus the above relations reduce the kernel $G_{\Q, \nu}(xt)$ in \eqref{78} as
\begin{align}\label{kjg}
G_{\Q, \nu}(xt)=\frac{1}{2\sin(\pi\nu)}\left\{\frac{2}{\pi}\sin(2\pi\nu)K_{2\nu}(2\sqrt{xt})+J_{-2\nu}(2\sqrt{xt})-J_{2\nu}(2\sqrt{xt})\right\}.
\end{align}
The definition of $Y$-Bessel function in \eqref{Bessel function} yields
\begin{align}\label{l1ydefn}
J_{-2\nu}(2\sqrt{xt})=\cos(2\pi\nu)J_{2\nu}(2\sqrt{xt})-\sin(2\pi\nu)Y_{2\nu}(2\sqrt{xt}).
\end{align}
We finally substitute \eqref{l1ydefn} into \eqref{kjg} to conclude the first part of our theorem. For the second part, the result follows directly for $(r_1, r_2) = (0, 1)$, by applying Proposition \ref{Slater} on the corresponding Meijer G-function.

\section{Koshliakov transform over any number field}\label{Kosh}
In this section, we generalize first Koshliakov transform \eqref{Koshliakov transform} over an arbitrary number field.
\subsection{Proof of Theorem \ref{Generalized Koshliakov transform}}
It follows from the Mellin transform of $K$-Bessel function \cite[p. 115, Equation (11.1)]{Ober} that
\begin{align*}
H_1(s):=\int_0^\infty K_\mu(t)t^{\mu+\nu}t^{s-1}\ dt = 2^{\mu+\nu+s-2}\Gamma\left(\frac{s+\nu}{2}\right)\Gamma\left(\frac{2\mu+\nu+s}{2}\right)
\end{align*}
where $\Re(s) > \pm \Re(\mu) - \Re(\mu+\nu)$.
The definition of Meijer G-function \eqref{G-function} yields for $c > \pm \Re(\nu)$
\begin{align*}
G_{\K, \nu}(xt)&=\frac{1}{2\pi i}\int_{(c)}\frac{\Gamma\left(-\frac{\nu}{2}+s\right)\Gamma\left(\frac{\nu}{2}+s\right)^{r_1+r_2}\Gamma\left(\frac{1+\nu}{2}+s\right)^{r_2}}{\Gamma\left(\frac{1+\nu}{2}-s\right)\Gamma\left(1-\frac{\nu}{2}-s\right)^{r_2}\Gamma\left(\frac{1-\nu}{2}-s\right)^{r_1+r_2}}\left(\frac{x^2}{16}\right)^{-s}t^{-2s}\ ds\nonumber\\
&=\frac{1}{4\pi i}\int_{(c/2)}\frac{\Gamma\left(-\frac{\nu}{2}+\frac{s}{2}\right)\Gamma\left(\frac{\nu}{2}+\frac{s}{2}\right)^{r_1+r_2}\Gamma\left(\frac{1+\nu}{2}+\frac{s}{2}\right)^{r_2}}{\Gamma\left(\frac{1+\nu}{2}-\frac{s}{2}\right)\Gamma\left(1-\frac{\nu}{2}-\frac{s}{2}\right)^{r_2}\Gamma\left(\frac{1-\nu}{2}-\frac{s}{2}\right)^{r_1+r_2}}\left(\frac{x^2}{16}\right)^{-\frac{s}{2}}t^{-s}\ ds,
\end{align*}
where in the last step we substitute $s$ by $s/2$. Therefore, the Mellin transform of $G_{\K, \nu}(xt)$ can be written as
\begin{align*}
H_2(s) :=\frac{1}{2}\frac{\Gamma\left(-\frac{\nu}{2}+\frac{s}{2}\right)\Gamma\left(\frac{\nu}{2}+\frac{s}{2}\right)^{r_1+r_2}\Gamma\left(\frac{1+\nu}{2}+\frac{s}{2}\right)^{r_2}}{\Gamma\left(\frac{1+\nu}{2}-\frac{s}{2}\right)\Gamma\left(1-\frac{\nu}{2}-\frac{s}{2}\right)^{r_2}\Gamma\left(\frac{1-\nu}{2}-\frac{s}{2}\right)^{r_1+r_2}}\left(\frac{x^2}{16}\right)^{-\frac{s}{2}}.
\end{align*}
It now follows from the conditions of hypothesis that the Perseval's formula \cite[p. 83]{Paris} evaluates the following integral along the vertical line $(d)$ satisfying $\pm\Re(\nu)<d<\min \lbrace 1+ \Re(\nu), 1+\Re(\nu + 2\mu) \rbrace$ as
\begin{align*}
\int_0^\infty K_\mu(t)t^{\mu+\nu}G_{\K, \nu}(xt) \ dt &=\frac{1}{2\pi i}\int_{(d)}H_1(1-s)H_2(s)\ ds =\frac{2}{2\pi i}\int_{(\frac{d}{2})}H_1(1-2s)H_2(2s)\ ds\nonumber\\
&=\frac{2^{\mu+\nu-1}}{2\pi i}\int_{(\frac{d}{2})}\tfrac{\Gamma\left(\frac{2\mu+\nu+1}{2}-s\right)\Gamma\left(-\frac{\nu}{2}+s\right)\Gamma\left(\frac{\nu}{2}+s\right)^{r_1+r_2}\Gamma\left(\frac{1+\nu}{2}+s\right)^{r_2}x^{-2s}}{\Gamma\left(1-\frac{\nu}{2}-s\right)^{r_2}\Gamma\left(\frac{1-\nu}{2}-s\right)^{r_1+r_2}4^{-s}} ds.
\end{align*}
Substituting $s$ by $-s$ in the above equation, we arrive at
\begin{align*}
\int_0^\infty K_\mu(t)t^{\mu+\nu}G_{\K, \nu}(xt)\ dt =\frac{2^{\mu+\nu-1}}{2\pi i}\int_{(-\frac{d}{2})}\tfrac{\Gamma\left(\frac{2\mu+\nu+1}{2}+s\right)\Gamma\left(-\frac{\nu}{2}-s\right)\Gamma\left(\frac{\nu}{2}-s\right)^{r_1+r_2}\Gamma\left(\frac{1+\nu}{2}-s\right)^{r_2}}{\Gamma\left(1-\frac{\nu}{2}+s\right)^{r_2}\Gamma\left(\frac{1-\nu}{2}+s\right)^{r_1+r_2}}\left(\frac{x^2}{4}\right)^s ds.
\end{align*}
Finally, the definition of Meijer G-function concludes our theorem. In particular, for $\mu = -\nu$, \eqref{Generalized self reciprocal} follows immediately from the above result.

\section{Transformation formula of the Lambert series associated to $\sigma_{\mathbb{K},a}(n)$}\label{Trans}
In this section, we mainly investigate the transformation of the series $\sum_{n=1}^\infty \sigma_{\mathbb{K},a}(n)e^{-ny}$ where $a$ and $y$ are any complex numbers. The following big-oh estimate for the Meijer $G$-function plays a crucial role in proving Theorem \ref{General Lambert series transformation} and Theorem \ref{Analytic continuation}.
\begin{lemma}\label{Big O of G}
Let $a$ and $y$ be any complex numbers. Then as $n\to \infty$, we have
\begin{align}\label{bigohg}
&\MeijerG*{d+1}{1}{1}{2d+1}{-\frac{a}{4}}{-\frac{a}{4},\left(\frac{a}{4}\right)_{r_1+r_2},\left(\frac{1}{2}+\frac{a}{4}\right)_{r_2};\left(\frac{a}{4}\right)_{r_2},\left(\frac{1}{2}+\frac{a}{4}\right)_{r_1+r_2}}{\frac{4\pi^{2(d+1)}n^2}{y^2D_{\mathbb{K}}^2}}\nonumber\\
&=\frac{(-1)^{r_1}\pi^{\frac{d}{2}}2^{-(a+2)d}}{\sin\left(\frac{\pi a}{2} \right)^{r_1 +r_2}\cos\left(\frac{\pi a}{2} \right)^{r_2}} \left(\frac{2\pi^{d+1}n}{yD_\K} \right)^{-\frac{a}{2}-2} \sum_{k=0}^m\frac{(-1)^k}{\Gamma(-1-a-2k)} \left( \frac{(2\pi)^{d+1}ne^{-\frac{i\pi d}{2}}}{yD_\K} \right)^{-2k}+O\left(\frac{1}{n^{2m+\frac{a}{2}+4}}\right).
\end{align}
\end{lemma}
\begin{proof}
For $1\leq r \leq p<q,\ 1\leq m\leq q$, and $|\arg(z)|\leq\rho\pi-\delta, \rho>0,\ \delta\geq0$, it follows from \cite[p.~179, Theorem 2]{Luke} that as $|z|\to\infty$,
\begin{align}\label{asyofg}
\MeijerG*{m}{r}{p}{q}{a_1,\cdots,a_p}{b_1,\cdots,b_q}{z}\sim \sum_{j=1}^r\exp(-i\pi(\nu+1)a_j)\Delta_q^{m,r}(j)E_{p,q}\left(z\exp(i\pi(\nu+1)||a_j\right),
\end{align}
where $\nu=q-m-r$,
\begin{align*}
E_{p,q}(z||a_j)&:=\frac{z^{a_j-1}\prod_{\ell=1}^q\Gamma(1+b_\ell-a_j)}{\prod_{\ell=1}^p\Gamma(1+a_\ell-a_j)}\sum_{k=0}^m\frac{\prod_{\ell=1}^q\left(1+b_\ell-a_j\right)_k}{k!\prod_{\substack{\ell=1\\\ell\neq j}}^p(1+a_\ell-a_j)_k}\left(-\frac{1}{z}\right)^k,\nonumber\\
\Delta_q^{m,r}(j)&:=(-1)^{\nu+1}\frac{\prod_{\substack{\ell=1\\ \ell\neq j}}^r\Gamma(a_\ell-a_j)\Gamma(1+a_\ell-a_j)}{\prod_{\ell=m+1}^q\Gamma(a_j-b_\ell)\Gamma(1+b_\ell-a_j)}.
\end{align*}
Thus the lemma follows for $m=d+1,\ r=p=1,\ q=2d+1$ in \eqref{asyofg} and using \eqref{Reflection formula} and \eqref{Duplication formula}.
\end{proof}
\subsection{Proof of Theorem \ref{General Lambert series transformation}}
We first prove the transformation formula for $0 < \Re(a) < 1$ and $y > 0$, later we extend it to $\Re(a) >-1$ and $Re(y) > 0$ respectively by analytic continuation. We consider the particular Schwartz function $f(n) = e^{-ny}$ with $y>0$ in Theorem \ref{zetazetathm}, which yields the following identity
\begin{align}\label{fisexp}
\sum_{n=1}^\infty \sigma_{\mathbb{K},a}(n)e^{-ny}&=-\frac{1}{2}\zeta_{\mathbb{K}}(-a)+\frac{\zeta_{\mathbb{K}}(1-a)}{y}+\frac{H\Gamma(a+1)\zeta(a+1)}{y^{a+1}}\nonumber\\
&\quad+2\pi^{\frac{{1+a+(1-a)d}}{2}}D_{\mathbb{K}}^{\frac{a-1}{2}}\sum_{n=1}^\infty \sigma_{\mathbb{K},-a}(n)n^{\frac{a}{2}}\int_0^\infty t^{\frac{a}{2}}e^{-yt}G_{\mathbb{K},a}\left(\frac{4\pi^{d+1}nt}{D_{\mathbb{K}}}\right)\ dt.
\end{align}
Letting $\mu=1/2,\ \nu=\frac{a}{2}$ and $x=\frac{4\pi^{d+1}n}{D_{\mathbb{K}}}$ in \eqref{Eqn:Generalized Koshliakov transform} and using the fact that $K_{\frac{1}{2}}(t)=\sqrt{\frac{\pi}{2t}}e^{-t}$, the integral on the right-hand side of \eqref{fisexp} evaluates as
\begin{align}\label{valueofint}
\int_0^\infty t^{\frac{a}{2}}e^{-yt}G_{\mathbb{K},a}\left(\frac{4\pi^{d+1}nt}{D_{\mathbb{K}}}\right)\ dt =\sqrt{\frac{2^a}{\pi}}y^{-\frac{a}{2}-1}\MeijerG*{d+1}{1}{1}{2d+1}{-\frac{a}{4}}{-\frac{a}{4},\left(\frac{a}{4}\right)_{r_1+r_2},\left(\frac{1}{2}+\frac{a}{4}\right)_{r_2};\left(\frac{a}{4}\right)_{r_2},\left(\frac{1}{2}+\frac{a}{4}\right)_{r_1+r_2}}{\frac{4\pi^{2(d+1)}n^2}{y^2D_{\mathbb{K}}^2}}.
\end{align}
We next substitute the evaluation \eqref{valueofint} into \eqref{fisexp} to arrive at our theorem for $0 < \Re(a) < 1$ and $y > 0$. It remains to show next that the result is also valid for $\Re(a) > -1$ and $Re(y) > 0$. Thus the bounds on Lemma \ref{Big O of G} and \eqref{Bound for sigma} together implies that the series on the right hand side of \eqref{Eqn:General Lambert series transformation} converges uniformly as long as $\Re(a) > -1$. The summand of the series is analytic for $\Re(a) > -1$, therefore by Weierstrass’ theorem on analytic functions it follows that the series represents an analytic function of $a$ when $\Re(a) > -1$.

On the other hand, the left-hand side of \eqref{Eqn:General Lambert series transformation} is analytic for $\Re(a) > -1$, thus by the principle of analytic continuation, \eqref{Eqn:General Lambert series transformation} holds for $\Re(a) > -1$ and $y > 0$. The both sides of \eqref{Eqn:General Lambert series transformation} are also analytic not only as a function of $y$ but also for $\Re(y) > 0$. Therefore again by applying the principle of analytic continuation, we can conclude our theorem.

\subsection{Proof of Theorem \ref{Analytic continuation}}
We add and subtract the finite sum
\begin{align*}
\frac{(-1)^{r_1}\pi^{\frac{d}{2}}2^{-(a+2)d}}{\sin\left(\frac{\pi a}{2} \right)^{r_1 +r_2}\cos\left(\frac{\pi a}{2} \right)^{r_2}} \left(\frac{2\pi^{d+1}n}{yD_\K} \right)^{-\frac{a}{2}-2} \sum_{k=0}^m\frac{(-1)^k}{\Gamma(-1-a-2k)} \left( \frac{(2\pi)^{d+1}ne^{-\frac{i\pi d}{2}}}{yD_\K} \right)^{-2k}
\end{align*}
from the summand on the right-hand side of \eqref{Eqn:General Lambert series transformation} to rewrite the transformation as
\begin{align}\label{beforefinite}
\sum_{n=1}^\infty \sigma_{\mathbb{K},a}(n)e^{-ny}&=-\frac{1}{2}\zeta_{\mathbb{K}}(-a)+\frac{\zeta_{\mathbb{K}}(1-a)}{y}+\frac{H\Gamma(a+1)\zeta(a+1)}{y^{a+1}}+\frac{2^{1+\frac{a}{2}}\pi^{\frac{a+(1-a)d}{2}}D_{\mathbb{K}}^{\frac{a-1}{2}}}{y^{1+\frac{a}{2}}}\nonumber\\
&\times\sum_{n=1}^\infty \sigma_{\mathbb{K},-a}(n)n^{\frac{a}{2}}\Bigg\{\MeijerG*{d+1}{1}{1}{2d+1}{-\frac{a}{4}}{-\frac{a}{4},\left(\frac{a}{4}\right)_{r_1+r_2},\left(\frac{1}{2}+\frac{a}{4}\right)_{r_2};\left(\frac{a}{4}\right)_{r_2},\left(\frac{1}{2}+\frac{a}{4}\right)_{r_1+r_2}}{\frac{4\pi^{2(d+1)}n^2}{y^2D_{\mathbb{K}}^2}}\nonumber\\
&-\frac{(-1)^{r_1}\pi^{\frac{d}{2}}2^{-(a+2)d}}{\sin\left(\frac{\pi a}{2} \right)^{r_1 +r_2}\cos\left(\frac{\pi a}{2} \right)^{r_2}} \left(\frac{2\pi^{d+1}n}{yD_\K} \right)^{-\frac{a}{2}-2} \sum_{k=0}^m\frac{(-1)^k}{\Gamma(-1-a-2k)} \left( \frac{(2\pi)^{d+1}ne^{-\frac{i\pi d}{2}}}{yD_\K} \right)^{-2k}\Bigg \}\nonumber\\
&+\frac{(-1)^{r_1}y(2\pi)^{-(a+2)d-1} \pi^{d-1}}{\sin\left(\frac{\pi a}{2} \right)^{r_1 +r_2}\cos\left(\frac{\pi a}{2} \right)^{r_2}}  \sum_{k=0}^m\frac{(-1)^k}{\Gamma(-1-a-2k)} \left( \frac{(2\pi)^{d+1} e^{-\frac{i\pi d}{2}}}{yD_\K} \right)^{-2k}\sum_{n=1}^{\infty}\frac{\sigma_{\K,-a}(n)}{n^{2k+2}}.\nonumber
\end{align}
Therefore, \eqref{continationeqn} follows directly for $\Re(a)>-1$ by applying \eqref{Dirichlet series of zeta zetak} in the last term of the above equation.

The bounds of Lemma \ref{Big O of G} and \eqref{Bound for sigma} together implies that the series on the right-hand side is uniformly convergent in $\Re(a)>-2m-3-\epsilon$ for any $\epsilon>0$. The summand of the series is analytic in this region of $a$, therefore by Weiestrass theorem for analytic functions, it follows that the series represents an analytic function of $a$ in  $\Re(a)>-2m-3$. It is straight forward to see that the other terms on the right-hand side of \eqref{continationeqn} are also analytic for $\Re(a)>-2m-3$.

On the other hand, the series on the left-hand side of \eqref{continationeqn} is an analytic function of $a$ for $\Re(a)>-2m-3$. Therefore, by the principle of analytic continuation, we can conclude that \eqref{continationeqn} holds for $\Re(a)>-2m-3$, which completes the proof of Theorem \ref{Analytic continuation}.


\begin{thebibliography}{00}

\bibitem{AAR}
G.~E.~Andrews, R.~Askey and R.~Roy, \emph{Special Functions}, Encyclopedia of
Mathematics and its Applications, 71, Cambridge University Press,  Cambridge,
1999.

\bibitem{Ayoub} R.~G.~Ayoub, {\em An introduction to the analytic theory of numbers}, American Math. Soc., (1963).

\bibitem{BK}
S.~Banerjee and R.~Kumar, {\em Explicit identities on zeta values over imaginary quadratic field}, submitted for publication.

\bibitem{bdrz}
B.~C.~Berndt, A.~Dixit, A.~Roy and A.~Zaharescu, \emph{New pathways and connections in number theory and analysis motivated by two incorrect claims of Ramanujan}, Advances in Mathematics {\bf 304} (2017), 809--929.

\bibitem{Bordelles} O.~Bordelles,  {\em Arithmetic tales}, London: Springer, 2012.
 
\bibitem{DKK} A.~Dixit, A.~Kesarwani and R.~Kumar, {\em A generalized modified Bessel function and explicit transformations of certain Lambert series}, submitted for publication.  

\bibitem{DKesarwani} 
A.~Dixit, A.~Kesarwani and V.~H.~Moll, \emph{A generalized modified Bessel function and a higher level analogue of the theta transformation formula} (with an appendix by N. M. Temme), J. Math. Anal. Appl. 459 (2018), 385–418.

\bibitem{DKumar} A.~Dixit and R.~Kumar, \emph{Superimposing theta structure on a generalized modular relation}, to appear in Research in the Mathematical Sciences.

\bibitem{DRoy} A.~Dixit and A. Roy, \emph{Analogue of a Fock-type integral arising from electromagnetism and its applications in number theory}, Res. Math. Sci. 7, 25 (2020), 1–33.

\bibitem{DF} A.~L.~Dixon and W.~L.~Ferrar, {\em Infinite integrals of Bessel functions}, Quart. J. Math. {\bf 1} (1935), 161--174.

\bibitem{erd1}
A.~Erdelyi, W.~Magnus, F.~Oberhettinger, and  F.~G.~Tricomi, \emph{Higher Transcendental Functions}, Vol. 1 (New York, 1953). Zentralblatt MATH, 51.

\bibitem{Hardy} G.~H.~Hardy, {\em On Dirichlet’s divisor problem}, Proc. Lond. Math. Soc. (2) 15 (1916), 1--25.

\bibitem{Huxley} M.~N.~Huxley, {\em Exponential sums and lattice points}, III, Proc. Lond. Math. Soc. (3) {\bf 87} (2003), 591--609.

\bibitem{Iwaniec} H.~Iwaniec and E.~Kowalski, {\em Analytic Number Theory}, Amer. Math. Soc. Colloquium Publ. {\bf 53}, Amer. Math. Soc., Providence, RI, (2004).

\bibitem{Koshliakov} N.~S.~Koshliakov, {\em On Voronoi's sum-formula}, Mess. Math. {\bf 58} (1929), 30--32.

\bibitem{Koshliakov2} N.~S.~Koshliakov, {\em Note on certain integrals involving Bessel functions}, Bull. Acad. Sci. URSS Ser. Math. {\bf 2} No. 4, 417--420; English text 421--425 (1938).

\bibitem{Lang} S.~Lang, {\em Algebraic number theory}, Addison-Wesley: Reading, MA, (1970).

\bibitem{Lau} Y.~K.~Lau, {\em On a generalized divisor problem I}, Nagoya Math. J. {\bf 165} (2002), 71--78.

\bibitem{Luke}
Y.~L.~Luke,  \emph{Special Functions and Their Approximations} v. {\bf 2}, Academic press, (1969). 

\bibitem{Ober} F.~Oberhettinger, \emph{Tables of Mellin Transforms}, Springer-Verlag, New York, 1974.

\bibitem{Paris} R.~B.~Paris and D.~Kaminski, {\em Asymptotics and Mellin-Barnes Integrals}, Encyclopedia of Mathematics and its Applications, {\bf 85}, Cambridge University Press, Cambridge, 2001.

\bibitem{Sound} K.~Soundararajan, {\em Omega results for the divisor and circle problems}, Int. Math. Res. Not. IMRN, {\bf 36} (2003), 1987--1998.

\bibitem{Voronoi} G.~Vorono\"{\dotlessi}, {\em Sur une fonction transcendante et ses applications \`a la sommation de quelques s\'eries}, Ann. Ecole Norm. Sup., {\bf 21} (1904), 207--267, 459--533.

\bibitem{watson-1944a}
G.~N.~Watson, \emph{A Treatise on the Theory of Bessel Functions}, second ed., Cambridge University Press, London, 1944.

\end{thebibliography}
\end{document}